\documentclass[12pt]{amsart}

\usepackage{amscd,amssymb,amsopn,amsmath,amsthm,mathrsfs,graphics,amsfonts,enumerate,verbatim,calc
}

\usepackage[all]{xy}

\usepackage{color}

\usepackage[OT2,OT1]{fontenc}
\usepackage{systeme}
\usepackage{tikz,tikz-cd}
\newcommand\cyr{%
\renewcommand\rmdefault{wncyr}%
\renewcommand\sfdefault{wncyss}%
\renewcommand\encodingdefault{OT2}%
\normalfont
\selectfont}
\DeclareTextFontCommand{\textcyr}{\cyr}

\usepackage{amssymb,amsmath}

\DeclareFontFamily{OT1}{rsfs}{}
\DeclareFontShape{OT1}{rsfs}{n}{it}{<-> rsfs10}{}
\DeclareMathAlphabet{\mathscr}{OT1}{rsfs}{n}{it}

\numberwithin{equation}{section}
\newtheorem{theorem}{Theorem}[section]
\newtheorem{lemma}[theorem]{Lemma}

\newtheorem{corollary}[theorem]{Corollary}

\newtheorem{question}{Question}

\theoremstyle{definition}
\newtheorem{definition}[theorem]{Definition}
\newtheorem{remark}[theorem]{Remark}
\theoremstyle{remark}

\newtheorem{acknowledgement}{Acknowledgement}

\renewcommand{\ker}{\operatorname{Ker}}

\newcommand{\fm}{\frak{m}}
\newcommand{\fp}{\frak{p}}
\newcommand{\fq}{\frak{q}}

\begin{document}
\title[The Frobenius closure of parameter ideals]{On the Frobenius closure of parameter ideals when the ring is F-injective on the punctured spectrum}

\author[Duong Thi Huong]{Duong Thi Huong}
\address{Department of Mathematics, Thang Long University, Hanoi, Vietnam}
\email{huongdt@thanglong.edu.vn}

\author[Pham Hung Quy]{Pham Hung Quy}
\address{Department of Mathematics, FPT University, Hanoi, Vietnam}
\email{quyph@fe.edu.vn}

\thanks{
The author is partially supported by a fund of Vietnam National Foundation for Science
and Technology Development (NAFOSTED)}


\begin{abstract} 
Let $(R,\fm)$  be an excellent generalized Cohen-Macaulay local ring of dimension $d$ that is $F$-injective on the punctured spectrum. Let $\frak q$ be a standard parameter ideal of $R$. The aim of the paper is to prove that 
$$\ell_R(\fq^F/\fq)\leq \sum\limits_{i=0}^{d}\binom{d}{i}\ell_R(0^F_{H^i_{\fm}(R)}).$$
Moreover, if $\frak q$ is contained in a large enough power of $\fm$, we have
$$\fq^F/\fq \cong \bigoplus_{i=0}^d (0^F_{H^i_{\fm}(R)})^{\binom{d}{i}}.$$
\end{abstract}
\maketitle
\section{Introduction}
Let $(R,\fm)$ be a Noetherian local ring of characteristic $p>0$ and of dimension $d$, $I$ an ideal of $R$ and $\fq$ a parameter ideal of $R$. The {\it Frobenius closure} of $I$ is $I^F = \{x \mid x^{p^e} \in I^{[p^e]} \text{ for some } e \ge 0\}$, where $I^{[p^e]} = (r^{p^e} \mid r \in I)$ is the $e$-th Frobenius power of $I$. The Frobenius endomorphism $F:R\to R$, $x\mapsto x^p$ and its localizations induce Frobenius action on the local cohomology module $F:H_{\fm}^i(R)\to H_{\fm^{[p]}}^i(R)\cong H_{\fm}^i(R)$ for all $i\ge 0$ via \v{C}ech complex. Similarly, we define the {\it Frobenius closure of the zero submodule} of $H_{\fm}^i(R)$ is 
$$0^F_{H^i_{\fm}(R)}=\{\eta\in H^i_{\fm}(R)\mid F^e(\eta)=0 \text{ for some } e \}.$$
Since $H^i_{\fm}(R)$ is always Artinian for all $i\ge 0$, there is a non-negative integer $e$ such that 
$$0^F_{H^i_{\fm}(R)}= \ker (H^i_{\fm}(R) \overset{F^{e'}}{\longrightarrow} H^i_{\fm}(R))$$ for all $e'\ge e$ by Hartshorne-Speiser \cite{HS77} and Lyubeznik \cite{L97}.  We define the {\it Hartshorne-Speiser-Lyubeznik number} of a local ring $R$ as follows
$$\mathrm{HSL}(R) = \min \{ e \mid   0^F_{H^i_{\fm}(R)} =   \ker (H^i_{\fm}(R) \overset{F^e}{\longrightarrow} H^i_{\fm}(R)) \text{ for all } i = 0, \ldots, d\}.$$
Due to the Noetherianess of $R$, for every ideal $I$ there exists a non-negative integer $e$ such that $(I^F)^{[p^e]}=I^{[p^e]}$. The smallest number $e$ satisfying the condition is called {\it Frobenius test exponent} of $I$, denoted by $\mathrm{Fte}(I)$. There is no uniform bound for the Frobenius test exponent of all ideals by Brenner \cite{B06}. To restrict to parameter ideals, we define the {\it Frobenius test exponent for parameter ideals} of the ring $R$ is
$$\mathrm{Fte}(R) = \min \{e\mid  (\fq^F)^{[p^e]}=\fq^{[p^e]} \text{ for all } \fq \}.$$
And conventionally $\mathrm{Fte}(R)=\infty$ if we have no such $e$.
Recall that $R$ is called {\it $F$-injective} if the Frobenius action on all local cohomology modules $H^i_{\fm}(R)$ are injective. Thus $R$ is $F$-injective iff $0^F_{H^i_{\fm}(R)}=0$ for all $i\ge 0$ that is equivalent to $\mathrm{HSL}(R)=0$. And the condition $\fq^F=\fq$ for all parameter ideals $\fq$ agrees with $\mathrm{Fte}(R)=0$. There are strong connections between $\fq^{F}$ and $0^{F}_{H^i_{\fm}(R)}$ for $i\ge 0$, while the $\mathrm{Fte}(R)$ is closely related with the $\mathrm{HSL}(R)$. Fedder proved that in a Cohen-Macaulay local ring $(R,\fm)$, two conditions $\fq^F=\fq$ for all parameter ideals $\fq$ and $0^{F}_{H^i_{\fm}(R)}=0$ for all $i\ge 0$ are equivalent. Moreover, Katzman and Sharp \cite{KS06} showed that $\mathrm{Fte}(R)=\mathrm{HSL}(R)< \infty$ in Cohen-Macaulay rings. Ma \cite{Ma15} extended Fedder's result for class of generalized Cohen-Macaulay local rings. If $R$ is generalized Cohen-Macaulay ring, Huneke, Katzman, Sharp and Yao proved that $\mathrm{Fte}(R)< \infty$. This assertion is reproved by the second author in \cite{Q19}: If $R$ is a generalized Cohen-Macaulay ring, then for any standard parameter ideal $\fq$ we have
$$\mathrm{Fte}(\fq) \le \sum_{i=0}^d\binom{d}{i} \mathrm{HSL}(H^i_{\fm}(R)).$$ 
In any local ring $(R,\fm)$, the second author and Shimomoto \cite{QS17} showed that if every parameter ideal is Frobenius closed then $R$ is $F$-injective. That is if $\mathrm{Fte}(R) = 0$ then $\mathrm{HSL}(R) = 0$. The converse is not true for non-equdimensional local rings. 
In \cite[Question 3]{Q18} the second author asked the following question. 
\begin{question}\label{Q question} 
Let $(R, \fm)$ be an excellent generalized Cohen-Macaulay local ring that is $F$-injective on the punctured spectrum. Is it true that
$$\ell_R (\fq^F / \fq) \le \sum_{i=0}^{d}\binom{d}{i} \ell_R (0^F_{H^i_{\fm}(R)})$$
for all standard parameter ideals $\fq$.\footnote{In Question 3 of \cite{Q18}, the second author missed the word "standard" in the statement.}
\end{question}
Fortunately, the method of \cite{Q19} works for this question. 
\begin{theorem}[=Theorem 3.3]
Question \ref{Q question} has an affirmative answer.
\end{theorem}
Moreover we also show a better result when $\fq$ is contained in a large enough power of $\fm$. This is based on the splitting of local cohomology in \cite{CQ11}.
\begin{theorem}[=Theorem 3.5]
Let $(R,\fm)$ be an excellent generalized Cohen-Macaulay local ring of dimension $d$ that is $F$-injective on punctured spectrum. Let $\frak q$ be a parameter ideal of $R$ contained in a large enough power of $\fm$. We have 
$$\fq^F/\fq \cong \bigoplus_{i=0}^d (0^F_{H^i_{\fm}(R)})^{\binom{d}{i}}.$$
\end{theorem}
Finally, we apply the above theorem to get a good bound for $\mathrm{Fte}(R)$ in the case $(R,\fm)$ is an excellent generalized Cohen-Macaulay local ring that is $F$-injective on the punctured spectrum.\\

 The paper is organized as follows. In the next section we collect the basic notion and background material relevant to the results of the last section, while the main results are presented in the last section.
\begin{acknowledgement} This paper was developed while the authors visited the Vietnam Institute for Advanced Study in Mathematics (VIASM). The authors are grateful to the VIASM for the financial support and hospitality.
\end{acknowledgement}
 
\section{Preliminary}
\subsection{Generalized Cohen-Macaulay module}
 In this subsection, the ring $R$ is free characteristic.
 We recall the notation of the filter regular sequence which was introduced by Cuong, Schenzel, and Trung in \cite{CST78}. Let $(R,\fm)$ be a Noetherian local ring, $M$ a finitely generated $R$-module and $x_1,\ldots, x_t$ a sequence of elements of $R$. Then we say that $x_1,\ldots, x_t$ is a {\it filter regular sequence} on $M$ if the quotient
$$\frac{(x_1,\ldots,x_{i-1})M:_Mx_i}{(x_1,\ldots,x_{i-1})M}$$
is an $R$-module of finite length for every $1\le i\le t$.
If the sequence $x_1,\ldots, x_t$ is a filter regular sequence on $M$ then the sequence $x_1^{n_1},\ldots,x_t^{n_t}$ is a filter regular sequence on $M$ for all $n_1,\ldots,n_t\geq 1$.
\begin{definition}
Let $(R,\fm)$ be a Noetherian local ring, $M$ a finitely generated $R$-module of dimension $t>0$, and $\fq=(x_1,\ldots, x_t)$ a parameter ideal of $M$. Then
\begin{enumerate}
\item $M$ is called {\it generalized Cohen-Macaulay} if $H^i_{\fm}(M)$ is finitely generated for all $i<t$.
\item The parameter ideal $\fq$ is called {\it standard} if $\fq H^j_{\fm}(M/(x_1,\ldots,x_i)M)=0$ for all $i+j<t$.
\end{enumerate}
\end{definition}
\begin{remark}
Let $(R,\fm)$ be a Noetherian local ring that is a homomorphic image of a Cohen-Macaulay local ring. Let $M$ be a finitely generated $R$-module of dimension $t>0$. Then
\begin{enumerate}
\item Module $M$ is generalized Cohen-Macaulay if and only if every system of parameters $x_1,\ldots,x_t$ of $M$ is a filter regular sequence on $M$.
\item Module $M$ is generalized Cohen-Macaulay if and only if $M$ is equidimensional and $M_{\fp}$ are Cohen-Macaulay for all $\fp \in \mathrm{Spec}(R)\setminus \{\fm\}$, that is $M$ is equidimensional and Cohen-Macaulay on the puncture spectrum. 
\end{enumerate}
\end{remark}
We need also a splitting result for local cohomology \cite[Corollary 4.11]{CQ11}.
\begin{lemma}\label{splitting} Let $M$ be a generalized Cohen-Macaulay module of dimension $t$. Let $n_0$ be a positive integer such that $\fm^{n_0}H^i_{\fm}(M) = 0$ for all $i < t$. Then for any parameter element $x \in \fm^{2n_0}$, we have
$$H^i_{\fm}(M/xM) \cong H^i_{\fm}(M) \oplus H^{i+1}_{\fm}(M)$$
for all $i < t-1$, and 
$$0:_{H^{\dim M-1}_{\fm}(M/xM)} \fm^{n_0} \cong H^{\dim M-1}_{\fm}(M) \oplus 0:_{H^{\dim M}_{\fm}(M)} \fm^{n_0}.$$ Moreover, if a system of parameter $x_1, \ldots, x_d$ is contained in $\fm^{2n_0}$, then it is standard.

\end{lemma}

\subsection{Positive characteristic $p$}
In this subsection, $(R,\fm)$ is a Noetherian ring of characteristic $p>0$ and of dimension $d$.\\

{\bf Frobenius closure.} Let $(R,\fm)$ be a Noetherian ring of characteristic $p>0$ and of dimension $d$. Let $F:R \xrightarrow{x \mapsto x^p} R$ be the Frobenius endomorphism and $F^e: R \xrightarrow{x \mapsto x^{p^e}} R$ the $e$-th Frobenius endomorphism. In order to notationally distinguish the source and the target of the $e$-th Frobenius endomorphism $F^e$, we will use $F_*^e(R)$ to denote the target. $F_*^e(R)$ has a left $R$-module structure via the $e$-th Frobenius endomorphism $F^e$, that is $r_1F_*^e(r_2):= F_*^e(r_1^{p^e}r_2)$ for all $r_1,r_2\in R$. 
%Define $R^{\circ}=R\setminus \cup_{\fp\in \mathrm{Min}(R)}\fp$ to be the complement of the union of all minimal primes of the ring $R$.
\begin{definition} Let $I$ be an ideal of $R$, we define
\begin{enumerate}
\item The {\it $e$-th Frobenius power} of $I$ is $I^{[p^e]} = (x^{p^e} \mid x \in I)$.
\item The {\it Frobenius closure} of $I$, $I^F = \{x \mid  x^{p^e} \in I^{[p^e]} \text{ for some } e \ge 0\}$.
\end{enumerate}
\end{definition}
Let $I$ be an ideal of the ring $R$. By the Noetherianess of $R$, there is an integer $e$, depending on $I$ such that $(I^F)^{[p^e]}=I^{[p^e]}$ for every ideal $I$. The smallest number $e$ satisfying the condition is called the {\it Frobenius test exponent of $I$}, and denoted by $\mathrm{Fte}(I)$,
$$\mathrm{Fte}(I)=\min \{e\mid (I^F)^{[p^e]}=I^{[p^e]}\}.$$
Katzman and Sharp showed the existence of an uniform bound for the Frobenius test exponents if we restrict to class of parameter ideals in a Cohen-Macaulay ring. We define the {\it Frobenius test exponent for parameter ideals} of $R$, $\mathrm{Fte}(R)$, the smallest integer $e$ satisfying the above condition $(\fq^F)^{[p^e]}=\fq^{[p^e]}$ for every parameter ideal $\fq$
$$\mathrm{Fte}(R) = \min \{e\mid  (\fq^F)^{[p^e]}=\fq^{[p^e]} \text{ for all } \fq \}.$$
And conventionally $\mathrm{Fte}(R)=\infty$ if we have no such $e$. Therefore it is natural to ask the following question. 
\begin{question}\label{Katzman Sharp} Let $(R, \fm)$ be an (equidimensional) local ring of positive characteristic $p$. Does there exist an uniform bound for Frobenius test exponents of parameter ideals, i.e., $\mathrm{Fte}(R)<\infty$?
\end{question}
It should be noted that the authors recently used the finiteness of $\mathrm{Fte}(R)$, if have, to find an upper bound of the multiplicity of a local ring \cite{HQ20}.\\

{\bf Frobenius action on local cohomology.} Let $x_1,\ldots, x_d$ be a system of parameters of $R$. Recall that local cohomology $H^i_{\fm}(R)$ may be computed as the homology of the \v{C}ech complex $\check{C}(x_1, \ldots, x_d;R)$
$$0 \longrightarrow R \longrightarrow \bigoplus_i R_{x_i} \longrightarrow \bigoplus_{i<j} R_{x_ix_j}   \longrightarrow \cdots  \longrightarrow R_{x_1 \ldots x_d} \to 0. $$
Consider the following diagram with \v{C}ech complexes $\check{C}(x_1, \ldots, x_d;R)$ and $\check{C}(x_1^p, \ldots, x_d^p;R)$ in two rows.
$$\begin{CD}
 0 @>>> R @>>>  \bigoplus_i R_{x_i}  \longrightarrow \cdots @>>> R_{x_1 \ldots x_d} @>>> 0\\
 @. @VFVV @VFVV @VFVV @.  \\
 0 @>>> R @>>>  \bigoplus_i R_{x_i^p}  \longrightarrow \cdots @>>> R_{x_1^p \ldots x_d^p} @>>> 0
\end{CD}
$$
where the vertical maps are the Frobenius endomorphism $F:R\to R$ and its localizations $F: R_u\to R_{u^{p}}$, $\frac{r}{u^s}\mapsto \frac{r^p}{{u^p}^s}$. The diagram induces a natural Frobenius action on local cohomology $F:H^i_{\fm}(R) \to H^i_{{\fm}^{[p]}}(R) \cong H^i_{{\fm}}(R)$ for all $i \ge 0$. We recall the definition of {\it $F$-injective} ring. 
\begin{definition}
A local ring $(R,\fm,k)$ is {\it $F$-injective} if the Frobenius action on $H^i_{\fm}(R)$ is injective for each $i\ge 0$.
\end{definition} 
\begin{definition}
For each $i\ge 0$, the {\it Frobenius closure of the zero submodule} of $H_{\fm}^i(R)$ is
 $$0^F_{H^i_{\fm}(R)}=\{\eta\in H^i_{\fm}(R)\mid F^e(\eta)=0 \text{ for some } e \}.$$
\end{definition} 
We have $0^F_{H^i_{\fm}(R)}$ is the nilpotent part of $H^i_{\fm}(R)$ under the Frobenius action.
Since $H^i_{\fm}(R)$ is always Artinian for all $i \ge 0$, by \cite[Proposition 1.11]{HS77}, \cite[Proposition 4.4]{L97} and \cite{Sh07}, there exists a non-negative integer $e$ such that $$0^F_{H^i_{\fm}(R)} = \ker (H^i_{\fm}(R) \overset{F^e}{\longrightarrow} H^i_{\fm}(R)).$$ 
We define the {\it Hartshorne-Speiser-Lyubeznik number} of a local ring $(R, \frak m)$ as follows
$$\mathrm{HSL}(R) = \min \{ e \mid   0^F_{H^i_{\fm}(R)} =   \ker (H^i_{\fm}(R) \overset{F^e}{\longrightarrow} H^i_{\fm}(R)) \text{ for all } i = 0, \ldots, d\}.$$
For Question \ref{Katzman Sharp}, Katzman and Sharp \cite{KS06} proved that $\mathrm{Fte}(R) = \mathrm{HSL} (R)$ provided $R$ is Cohen-Macaulay. In general, we have $\mathrm{Fte}(R)\geq \mathrm{HSL}(R)$ for any local ring in \cite{HQ19}. The result was extended for generalized Cohen-Macaulay case by Huneke, Katzman, Sharp, and Yao \cite{HKSY06}. In 2019, the second author \cite{Q19}, using the relative Frobenius action on local cohomology, proved $\mathrm{Fte}(R)<\infty$ for {\it weakly F-nilpotent} rings, i.e., $H^i_{\fm}(R) = 0^F_{H^i_{\fm}(R)}$ for all $i<d$. Maddox \cite{M19} showed this result for {\it generalized weakly F-nilpotent} rings, i.e., $H^i_{\fm}(R)/0^F_{H^i_{\fm}(R)}$ has finite length for all $i<d$. Very recently, this result is extended by the authors in \cite{HQ22}.\\

{\bf Relative Frobenius action on local cohomology.} The relative Frobenius action on local cohomology which was introduced in \cite{PQ19} by Polstra and the second author in study $F$-nilpotent rings. Let $K\subseteq I$ be ideals of $R$. The Frobenius endomorphism $F: R/K \to R/K$ can be factored as composition of two natural maps:
$$\begin{tikzcd}[column sep=10ex]
R/K \arrow[rr, "F" ] & & R/K\\
& R/K^{[p]} \arrow[ul,leftarrow, "F_R"] \arrow[ur,twoheadrightarrow, "\pi"'] & 
\end{tikzcd}$$
where the second map $\pi$ is the natural project map. We denote the first map by $F_R:R/K\to R/K^{[p]}$, $F_R(a+K)=a^p+K^{[p]}$ for all $a\in R$. The homomorphism $F_R$ induces {\it the relative Frobenius actions on local cohomology} $F_R: H^i_I(R/K)\to H^i_I(R/K^{[p]})$ via \v{C}ech complexes. We define the {\it relative Frobenius closure of the zero submodule of $H^i_I(R/K)$ with respect to $R$} as follows
$$0^{F_R}_{H^i_I(R/K)}=\{\eta \mid F^e_R(\eta)=0\in H^i_I(R/K^{[p^e]})\text{ for some }e\gg 0\}.$$
\section{Results}
In this section, $R$ is of positive characteristic $p$. We begin this section with a condition when $0^F_{H^i_{\fm}(R)}$ has finite length for every $i\geq 0$.
\begin{lemma}\label{Lemma 3.1}
Let $(R,\fm)$ be an excellent generalized Cohen-Macaulay local ring of dimension $d$ that is $F$-injective on the punctured spectrum. Then $\ell_R(0^F_{H^i_{\fm}(R)})<\infty$ for all $i\geq 0$.
\end{lemma}
\begin{proof}
We can pass to the $\frak m$-adic completion to assume that $R$ is complete. Since $R$ is generalized Cohen-Macaulay then $\ell_R(0^F_{H^i_{\fm}(R)})<\infty$ for all $i<d$. For $i=d$, we consider an exact for all $e\gg 0$
$$0\to 0^F_{H^d_{\fm}(R)}\to H^d_{\fm}(R)\xrightarrow{F^e} H^d_{\fm}(R).$$
Because $R$ is complete, by Cohen's structure theorem we can write $R=S/I$ where $S$ is a complete local ring. By local duality theorem, $H^d_{\fm}(R)^{\vee}\cong \mathrm{Ext}_S^{n-d}(R,S)$ where $n=\mathrm{dim}(S)$. The above exact induces the following exact
$$\mathrm{Ext}_S^{n-d}(R,S)\xrightarrow{(F^e)^{\vee}}\mathrm{Ext}_S^{n-d}(R,S)\to (0^F_{H^d_{\fm}(R)})^{\vee}\to 0.$$
Localization at any prime ideal $P\neq \fm$, for simplification we also denote $P$ for the pre-image of $P$ in $S$ and we have
$$\mathrm{Ext}_{S_P}^{n-d}(R_P,S_P)\xrightarrow{(F^e_P)^{\vee}}\mathrm{Ext}_{S_P}^{n-d}(R_P,S_P)\to (0^F_{H^d_{\fm}(R)})^{\vee}_P\to 0.$$
It should be noted that $n-d=\mathrm{dim}(S_P)-\mathrm{dim}(R_P)$ and now by local duality over $S_P$, we obtain the following exact
$$0\to ((0^F_{H^d_{\fm}(R)})^{\vee}_P)^{\vee_P}\cong 0^{F_P}_{H^{\dim R_P}_{PR_P}(R)} \to H_{PR_P}^{\mathrm{dim}R_P}(R_P)\xrightarrow{F^e_P}H_{PR_P}^{\mathrm{dim}R_P}(R_P).$$
Since $R_P$ is Cohen-Macaulay and $F$-injective so $\mathrm{Ker}(F^e_P)=0^F_{H_{PR_P}^{\mathrm{dim}R_P}(R_P)}=0$. Hence $(0^F_{H^d_{\fm}(R)})^{\vee}_P=0$ for all $P\neq \fm$, so $\mathrm{Supp}_R(0^F_{H^d_{\fm}(R)})\subseteq \{\fm\} $. Therefore $\ell_R(0^F_{H^d_{\fm}(R)})<\infty$.
\end{proof}
\begin{lemma}\label{Fro0} 
Let $(R, \fm)$ be a Noetherian local ring and $I$ an ideal of $R$. Then 
$$0^{F_R}_{H_{\fm}^0(R/I)}=I^F\cap (I:\fm^{\infty})/I.$$
\end{lemma}
\begin{proof}
Pick $x+I\in 0^{F_R}_{H_{\fm}^0(R/I)}$. We have $x\in I:\fm^{\infty}$ and there exists an integer $e$ such that $F^e(x+I)=x^{p^e}+I^{[p^e]}=0\in R/I^{[p^e]}$. Equivalently, $x+I\in I^F\cap (I:\fm^{\infty})/I $.
\end{proof}
The following theorem is an affirmative answer for \cite[Question 3]{Q18}. The idea of proof is the same with the proof of main result of \cite{Q19}.
\begin{theorem}
Let $(R,\fm)$ be an excellent generalized Cohen-Macaulay local ring of dimension $d$ that is $F$-injective on the punctured spectrum. Let $\frak q$ be a standard parameter ideal of $R$. Then 
$$\ell_R(\fq^F/\fq)\leq \sum\limits_{i=0}^{d}\binom{d}{i}\ell_R(0^F_{H^i_{\fm}(R)}).$$

\end{theorem}
\begin{proof}
Let  $\fq=(x_1,\ldots,x_d)$. Since $R$ is generalized Cohen-Macaulay we have $x_1,\ldots,x_d$ is a filter regular sequence. Set $\fq_i=(x_1,\ldots,x_i)$ for all $i=1,\ldots, d$, $\fq_0=(0)$. The short exact
$$0\to R/\fq_{i-1}:x_i\xrightarrow{.x_i}R/\fq_{i-1}\to R/\fq_i\to 0$$
induces a long exact for all $j\ge 0$
$$ H^j_{\fm}(R/\fq_{i-1}:x_i)\xrightarrow{.x_i} H^j_{\fm}(R/\fq_{i-1})\to H^j_{\fm}(R/\fq_i)\to H^{j+1}_{\fm}(R/\fq_{i-1}:x_i)\xrightarrow{.x_i} H^{j+1}_{\fm}(R/\fq_{i-1}).$$
Since $\fq$ is standard, $\fq H^j_{\fm}(R/\fq_i)=0$ for all $i+j<d$ so the map $ H^j_{\fm}(R/\fq_{i-1}:x_i)\xrightarrow{.x_i} H^j_{\fm}(R/\fq_{i-1})$ is zero map for all $i+j\leq d$. Therefore, for all $i+j<d$ we have the following exacts
$$0\to H^j_{\fm}(R/\fq_{i-1})\to H^j_{\fm}(R/\fq_i)\to H^{j+1}_{\fm}(R/\fq_{i-1})\to 0.$$
$$0\to H^{d-i}_{\fm}(R/\fq_{i-1})\to H^{d-i}_{\fm}(R/\fq_i)\to (0:x_i)_{H^{d-i+1}_{\fm}(R/\fq_{i-1})}\to 0.$$ 
When $i+j<d$, from the commutative diagram
$$
\begin{CD}
0\longrightarrow H_{\fm}^j(R/\fq_{i-1}) @>>f> H_{\fm}^j( R/\fq_{i})  @>>g>H_{\fm}^{j+1}( R/q_{i-1}) @>>> 0 \\
@VF^e_R VV @VF^e_RVV @VF^e_RVV  (\star )\\
0\longrightarrow H_{\fm}^j(R/\fq_{i-1}^{[p^e]}) @>>f'> H_{\fm}^j(R/\fq_{i}^{[p^e]}) @>>g'> H_{\fm}^{j+1}(R/\fq_{i-1}^{[p^e]}) @>>> 0
\end{CD}
$$
we have the following exact
$$0\to 0^{F_R}_{H_{\fm}^j(R/\fq_{i-1})}\xrightarrow{\overline{f}}0^{F_R}_{H_{\fm}^j(R/\fq_{i})}\xrightarrow{\overline{g}} 0^{F_R}_{H_{\fm}^{j+1}(R/\fq_{i-1})},$$
where $\overline{f}=f\mid_{0^{F_R}_{H_{\fm}^j(R/\fq_{i-1})}}$, $\overline{g}=g\mid_{0^{F_R}_{H_{\fm}^j(R/\fq_{i})}}$.
The exactness is due to the injection of $f$ and $f'$. Therefore for all $i+j<d$ we have
$$\ell_R(0^{F_R}_{H^j_{\fm}(R/\fq_i)})\leq \ell_R(0^{F_R}_{H^j_{\fm}(R/\fq_{i-1})})+ \ell_R(0^{F_R}_{H^{j+1}_{\fm}(R/\fq_{i-1})}).$$
Similarly when $i+j=d$ we have $$\ell_R(0^{F_R}_{H^{d-i}_{\fm}(R/\fq_i)})\leq \ell_R(0^{F_R}_{H^{d-i}_{\fm}(R/\fq_{i-1})})+ \ell_R(0^{F_R}_{H^{d-i+1}_{\fm}(R/\fq_{i-1})}).$$
To summarize for all $i+j\leq d$ we have
 $$\ell_R(0^{F_R}_{H^j_{\fm}(R/\fq_i)})\leq \ell_R(0^{F_R}_{H^j_{\fm}(R/\fq_{i-1})})+ \ell_R(0^{F_R}_{H^{j+1}_{\fm}(R/\fq_{i-1})}).$$
By induction on $k$ we have
$$\ell_R(0^{F_R}_{H^0_{\fm}(R/\fq_d)})\leq \sum\limits_{i=0}^{k}\binom{k}{i}\ell_R(0^{F_R}_{H^i_{\fm}(R/\fq_{d-k})}).$$
For $k=d$ we obtain
 $$\ell_R(0^{F_R}_{H^0_{\fm}(R/\fq_d)})\leq \sum\limits_{i=0}^{d}\binom{d}{i}\ell_R(0^{F}_{H^i_{\fm}(R)}).$$
By Lemma \ref{Fro0}, $0^{F_R}_{H^0_{\fm}(R/\fq_d)}=\fq^F/\fq$ and we have
    $$\ell_R(\fq^F/\fq)\leq \sum\limits_{i=0}^{d}\binom{d}{i}\ell_R(0^{F}_{H^i_{\fm}(R)}).$$
The proof is completed.
\end{proof}
\begin{remark}
Let $(R, \fm)$ be an excellent generalized Cohen-Macaulay local ring and $F$-injective on the punctured spectrum. By the same technique we can prove that for all parameter ideals $\fq$ we have
$$\ell_R(\fq^F/\fq)\leq \sum\limits_{i=0}^{d-1}\binom{d}{i}\ell_R(H^i_{\fm}(R)) + \ell(0^{F}_{H^d_{\fm}(R)}).$$
\end{remark}
We next consider parameter ideal contained in a lager power of $\frak m$.
\begin{theorem}\label{MainThm2} 
Let $(R,\fm)$ be an excellent generalized Cohen-Macaulay local ring of dimension $d$ that is $F$-injective on the punctured spectrum. Let $n_0$ be a positive integer such that $\fm^{n_0} H^i_{\fm}(R) = 0$ for all $i < d$, and $\fm^{n_0} 0^F_{H^d_{\fm}(R)} = 0$. Let $\frak q$ be a parameter ideal of $R$ contained in $\fm^{2n_0}$. Then 
$$\fq^F/\fq \cong \bigoplus_{i=0}^d (0^F_{H^i_{\fm}(R)})^{\binom{d}{i}}.$$
\end{theorem}
\begin{proof} Let $\fq = (x_1, \ldots, x_d)$. Since
$$H^i_{\fm}(R/(x_1)) \cong H^i_{\fm}(R) \oplus H^{i+1}_{\fm}(R)$$
for all $i < d-1$ we have $0^{F_R}_{H^i_{\fm}(R/(x_1))} \cong 0^F_{H^i_{\fm}(R)} \oplus 0^F_{H^{i+1}_{\fm}(R)}$ for all $i < d-1$. On the other hand since
 $\fm^{n_0} 0^F_{H^d_{\fm}(R)} = 0$ we have $0^F_{H^d_{\fm}(R)} \subseteq 0:_{H^d_{\fm}(R)} \fm^{n_0}$. The direct summand 
$$0:_{H^{d-1}_{\fm}(R/(x_1))} \fm^{n_0} \cong H^{d-1}_{\fm}(R) \oplus 0:_{H^d_{\fm}(R)} \fm^{n_0}$$ implies that $0^{F_R}_{H^{d-1}_{\fm}(R/(x_1))} \subseteq 0:_{H^{d-1}_{\fm}(R/(x_1))} \fm^{n_0}$, and $0^{F_R}_{H^{d-1}_{\fm}(R/(x_1))} \cong 0^F_{H^{d-1}_{\fm}(R)} \oplus 0^F_{H^{d}_{\fm}(R)}$. Therefore
$$0^{F_R}_{H^i_{\fm}(R/(x_1))} \cong 0^F_{H^i_{\fm}(R)} \oplus 0^F_{H^{i+1}_{\fm}(R)}$$ 
for all $i \le d-1$. Continue this progress we have 
$$0^{F_R}_{H^j_{\fm}(R/(x_1, \ldots, x_i))} \cong \bigoplus_{k=j}^{i+j} (0^F_{H^k_{\fm}(R)})^{\binom{i}{k-j}}$$
for all $i + j \le d$. Moreover $\fq^F/\fq \cong 0^{F_R}_{H^0_{\fm}(R/(x_1, \ldots, x_d))} $ so we have 
$$\fq^F/\fq \cong \bigoplus_{i=0}^d (0^F_{H^i_{\fm}(R)})^{\binom{d}{i}}.$$
The proof is completed.
\end{proof}
Let $\fq$ be a parameter ideal of $R$, and $e$ an integer. Notice that $(\fq^F)^{[p^e]}\subseteq (\fq^{[p^e]})^F$, so $\mathrm{Fte}(\fq)\leq \mathrm{Fte}(\fq^{[p^e]})+e$. Therefore we have the following corollary as a consequence.
\begin{corollary} Let $(R,\fm)$ be an excellent generalized Cohen-Macaulay local ring of dimension $d$ that is $F$-injective on punctured spectrum. Let $n_0$ be a positive integer such that $\fm^{n_0} H^i_{\fm}(R) = 0$ for all $i < d$, and $\fm^{n_0} 0^F_{H^d_{\fm}(R)} = 0$. We have
$$\mathrm{Fte}(R) \le \lceil  \log_p(2n_0)\rceil + \mathrm{HSL}(R),$$
where $\lceil x\rceil$ is the smallest integer that is greater than or equal to $x$. 
\end{corollary}
\begin{proof}
If parameter ideal $\fq\subseteq \fm^{2n_0}$, by Theorem \ref{MainThm2} we have
 $$\fq^F/\fq \cong \bigoplus_{i=0}^d (0^F_{H^i_{\fm}(R)})^{\binom{d}{i}}.$$
 Applying $F^{\mathrm{HSL}(R)}(-)$ to the above equivalence, we see that $F^{\mathrm{HSL}(R)}(0^F_{H^i_{\fm}(R)})=0$ for all $i\ge 0$ and thus $(\fq^F)^{[p^{\mathrm{HSL}(R)}]}=\fq^{[p^{\mathrm{HSL}(R)}]}$. Hence, $\mathrm{Fte(\fq)}\leq \mathrm{HSL}(R)$ for every parameter ideal $\fq\subseteq \fm^{2n_0}$. 
Take any parameter ideal $\fq$, set $e_0=\lceil  \log_p(2n_0)\rceil$, then $\fq^{[p^{e_0}]}\subseteq \fq^{2n_0}\subseteq \fm^{2n_0}$ and thus  $\mathrm{Fte}(\fq)\leq \mathrm{Fte}(\fq^{[p^{e_0}]})+e_0$.
Combination all above observation together we complete the proof. 
\end{proof}
We suspect the assumption of $F$-injective on the punctured spectrum not necessary.
\begin{question} Let $(R,\fm)$ be an excellent generalized Cohen-Macaulay local ring of dimension $d$ of characteristic $p$. Let $n_0$ be a positive integer such that $\fm^{n_0} H^i_{\fm}(R) = 0$ for all $i < d$. Is it true that
$$\mathrm{Fte}(R) \le \lceil  \log_p(2n_0)\rceil + \mathrm{HSL}(R),$$
where $\lceil x\rceil$ is the smallest integer that is greater than or equal to $x$. 
\end{question}

\end{document}